\documentclass{ifacconf}

\usepackage{graphicx}      
\usepackage{natbib}        
\usepackage{tikz}
\usepackage{amsmath}
\usepackage{amssymb}
\usepackage{mathtools}
\usepackage{xcolor}
\usepackage{enumerate}
\usepackage{subcaption}
\usepackage{dblfloatfix}
\usepackage{afterpage}

\definecolor{darkblue}{RGB}{0,0,255}
\definecolor{lightblue}{RGB}{77,187,230}
\definecolor{darkgreen}{RGB}{46,150,64}
\definecolor{lightgreen}{RGB}{136,231,136}

\DeclareMathOperator{\Rank}{rank}

\usepackage[customcolors]{hf-tikz}
\newcommand{\usegray}{
\hfsetbordercolor{gray!0}
\hfsetfillcolor{gray!15}
}
\newcommand{\usegraydark}{
\hfsetbordercolor{gray!0}
\hfsetfillcolor{gray!45}
}

\newcommand\scalemath[2]{\scalebox{#1}{\mbox{\ensuremath{\displaystyle #2}}}}
\begin{document}
\begin{frontmatter}

\title{Testing Backward-Flatness of Nonlinear Discrete-Time Systems} 

\thanks[footnoteinfo]{This research was funded in whole, or in part, by the Austrian Science Fund (FWF) P36473. For the purpose of open access, the author has applied a CC BY public copyright license to any Author Accepted Manuscript version arising from this submission.}

\author[First]{Johannes Schrotshamer} 
\author[First]{Bernd Kolar} 
\author[First]{Markus Sch{\"o}berl}

\address[First]{Institute of Control Systems, Johannes Kepler University Linz, Altenberger Strasse 69, 4040 Linz, Austria (e-mail: johannes.schrotshamer@jku.at, bernd.kolar@jku.at, markus.schoeberl@jku.at).}

\begin{abstract}                
Despite ongoing research, testing the flatness of discrete-time systems remains a challenging problem. To date, only the property of forward-flatness --- a special case of difference-flatness --- can be checked in a computationally efficient manner. In this paper, we propose a systematic approach for testing backward-flatness, which is another special case of difference-flatness, and for deriving a corresponding backward-flat output. Additionally, we discuss the relationship between the Jacobian matrices associated with the flat parameterization of backward- and forward-flat systems and illustrate our results by an academic example.
\end{abstract}


\end{frontmatter}

\section{Introduction}
The concept of differential flatness, introduced by Fliess, L{\'e}vine, Martin and Rouchon in the 1990s (see e.g. \cite{FliessLevineMartinRouchon:1995,FliessLevineMartinRouchon:1999}), has attracted a lot of interest within the control systems community. In simple terms, a continuous-time nonlinear system is flat if all system variables can be expressed by a flat output and its time derivatives, which itself depends on the system variables and their time derivatives. This implies the existence of a one-to-one correspondence of solutions of a flat system and the solutions of a trivial system. The popularity stems from the fact that flatness allows an elegant systematic solution of feedforward and feedback problems, making it highly relevant for practical applications. However, currently there do not exist verifiable necessary and sufficient conditions for differential flatness, except for certain classes of systems. For recent developments, see e.g. \cite{SchoeberlRiegerSchlacher:2010}, \cite{Schoeberl:2014}, \cite{NicolauRespondek:2017}, \cite{GstoettnerKolarSchoeberl:2022b,GstottnerKolarSchoberl:2021b}, or \cite{NicolauRespondekLi:2025}.\\
The discrete-time analogue of differential flatness --- known as difference flatness --- is commonly defined by simply replacing the time derivatives of the continuous-time definition by forward-shifts (see e.g. \cite{Sira-RamirezAgrawal:2004}, \cite{KaldmaeKotta:2013}, or \cite{KolarKaldmaeSchoberlKottaSchlacher:2016}).
For these so-called forward-flat systems, a geometric test in the form of necessary and sufficient conditions has been derived in \cite{KolarDiwoldSchoberl:2023}. Interestingly, no such counterpart exists for continuous-time systems.
Since this definition does not encompass all known discrete-time flat systems, research has been directed toward developing a more general notion of difference flatness.
Recently, in \cite{GuillotMillerioux2020}, a definition in which the flat output may additionally depend on backward-shifts of the input variables has been proposed. An even more general definition has been introduced in \cite{DiwoldKolarSchoberl:2020}, where the authors focused on the one-to-one correspondence between the solutions of a flat system and those of a trivial system. Within this approach, the flat output may depend not only on the system variables and their forward-shifts as well as backward-shifts of the input, but also on backward-shifts of the state.
For this class of discrete-time systems, necessary and sufficient conditions have been derived in \cite{Kaldmae:2021} using an algebraic approach. However, the presence of degrees of freedom and the need to solve partial differential equations lead to high computational complexity, motivating the development of alternative methods.\\
Based on this more general notion of difference flatness, it has been shown in \cite{SchrotshamerKolarSchoeberl:2025} that analogously to the special case of forward-flatness another subclass --- referred to as backward-flat systems --- emerges naturally when neither the flat output nor the flat parameterization involves forward-shifts.

While a computationally efficient test currently exists only for the property of forward-flatness, in the present contribution we address the existing gap for the recently introduced subclass of backward-flat systems by presenting a systematic approach to check discrete-time systems for backward-flatness. We prove that a discrete-time system is backward-flat if and only if an associated system is forward-flat. This associated system is constructed such that the trajectories of the system variables are in one-to-one correspondence with those of the original system. Consequently, the backward-flatness of a discrete-time system can be tested indirectly by applying the existing geometric test for forward-flatness to its associated system.\\
Additionally, if the associated system is forward-flat, a corresponding flat output can be derived as in \cite{KolarDiwoldSchoberl:2023}. Based on the forward-flat output of the associated system, a backward-flat output of the original system can then be obtained systematically.
Finally, we discuss the relationship between the Jacobian matrices of the flat parameterization of associated systems, based on rank conditions for specific submatrices.

The paper is organized as follows: In Section \ref{sec:two}, we first recall the concept of flatness for discrete-time systems, including the special cases of forward- and backward-flatness, as discussed in \cite{SchrotshamerKolarSchoeberl:2025}. Next, in Section \ref{sec:three}, we present associated systems whose trajectories are in a one-to-one correspondence with those of the original system and analyze their flatness. Based on the result that a discrete-time system is backward-flat if and only if its associated system is forward-flat, we give a systematic approach to test a discrete-time system for backward-flatness and for deriving a corresponding backward-flat output in Section \ref{sec:four}. Finally, we examine the relationship between the Jacobian matrices of the flat parameterization of associated systems and illustrate our results by an academic example.
\paragraph*{Notation}
    To enhance readability, we commonly denote the one-fold forward-shift of a variable by a superscript +, whereas backward- and higher-order forward-shifts are generally represented using subscripts in brackets, e.g. $u_{[\alpha]}$ denotes the $\alpha$th forward-shift and $u_{[-\alpha]}$ denotes the $\alpha$th backward-shift of $u$ with $\alpha\geq0$.
    If a variable $y$ consists of several components, i.e. $y^j=(y^1,\ldots,y^m)$ with $j=1,\ldots,m$, shifts are denoted with a multi-index. Let $S_1=(s_{1,1},\ldots,s_{1,m})$ and $S_2=(s_{2,1},\ldots,s_{2,m})$ be two multi-indices, then $y_{[-S_1]}$ denotes the backward-shifts and $y_{[S_2]}$ denotes the forward-shifts of the individual components of $y$. Moreover, $y_{[-S_1,S_2]}$ is an abbreviation for $y$ and its backward-shifts up to order $S_1$ and its forward-shifts up to order $S_2$.

\section{Discrete-Time Systems and Flatness} \label{sec:two}

    \subsection{Flatness of Discrete-Time Systems}

    In this contribution, we consider nonlinear time-invariant discrete-time systems of the form
	\begin{equation}\label{eq:sysEq}
		x^{i,+}=f^{i}(x,u)\,,\quad i=1,\dots,n
	\end{equation}
	with $\dim(x)=n$, $\dim(u)=m$ and smooth functions $f^{i}(x,u)$ that meet $\Rank (\partial_u f)=m$, i.e. independent inputs. In addition, we assume
    that the system equations \eqref{eq:sysEq} satisfy the submersivity condition 
    \begin{equation}\label{eq:submersivity_condition}
        \Rank(\partial_{(x,u)}f)=n\,.
    \end{equation}
    Submersivity is a common assumption in the discrete-time literature, as it is necessary for accessibility (see e.g. \cite{Grizzle:1993}).
    Furthermore, due to \eqref{eq:submersivity_condition}, there always exist $m$ functions $g(x, u)$ such that the map
    \begin{equation}\label{eq:sysEq_ext}
		x^{+} = f(x, u)\,,\quad \zeta = g(x, u)
	\end{equation}
	is locally a diffeomorphism with its inverse given by
    \begin{equation}\label{eq:sysEq_ext_inv}
		x=\psi_x(x^+,\zeta)\,, \quad u=\psi_u(x^+,\zeta)\,.
	\end{equation}
    Like in \cite{DiwoldKolarSchoberl:2020}, we refer to a discrete-time system \eqref{eq:sysEq} as being flat if there exists a one-to-one correspondence between its solutions $(x(k),u(k))$ and solutions $y(k)$ of a trivial system (arbitrary trajectories that need not satisfy any difference equation) with the same number of inputs.
    As illustrated in Fig.~\ref{fig:One_to_one_correspondence}, this correspondence means that the values of $x(k)$ and $u(k)$ at a given time step may depend on an arbitrary but finite number of past and future values of $y(k)$, and conversely, $y(k)$ may depend on an arbitrary but finite number of past and future values of $x(k)$ and $u(k)$, as formally stated in Definition 1 of \cite{DiwoldKolarSchoberl:2020}.

    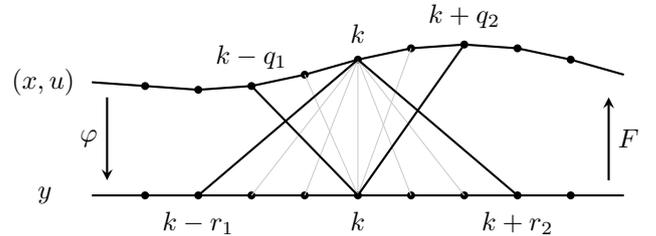
\begin{figure}[h!]
      \centering
      \begin{tikzpicture}[scale=1]
        
        \draw[thick] (-3.5,1.5) -- (-2.8,1.45) -- (-2.1,1.4) -- (-1.4,1.45) -- (-0.7,1.6) -- (0,1.8) -- (0.7,1.95) -- (1.4,2) -- (2.1,1.95) -- (2.8,1.8) -- (3.5,1.6);
        
        \draw[thick] (-3.5,0) -- (3.5,0);
        
        \fill (-2.8,1.45) circle (1.5pt);
        \fill (-2.1,1.4) circle (1.5pt);
        \fill (-1.4,1.45) circle (1.5pt);
        \fill (-0.7,1.6) circle (1.5pt);
        \fill (0,1.8) circle (1.5pt);
        \fill (0.7,1.95) circle (1.5pt);
        \fill (1.4,2) circle (1.5pt);
        \fill (2.1,1.95) circle (1.5pt);
        \fill (2.8,1.8) circle (1.5pt);
        
        \foreach \x in {-2.8,-2.1,-1.4,-0.7,0,0.7,1.4,2.1,2.8} {
            \fill (\x,0) circle (1.5pt);
        }
        
        \draw[gray!40] (0,0) -- (-0.7,1.6);
        \draw[gray!40] (0,0) -- (0,1.8);
        \draw[gray!40] (0,0) -- (0.7,1.95);
        \draw[gray!40] (-1.4,0) -- (0,1.8);
        \draw[gray!40] (-0.7,0) -- (0,1.8);
        \draw[gray!40] (0.7,0) -- (0,1.8);
        \draw[gray!40] (1.4,0) -- (0,1.8);
        
        \draw[thick] (0,0) -- (-1.4,1.45);
        \draw[thick] (0,0) -- (1.4,2);
        \draw[thick] (-2.1,0) -- (0,1.8);
        \draw[thick] (2.1,0) -- (0,1.8);
        
        \node[above] at (-1.4,1.55) {$k-q_1$};
        \node[above] at (0,1.9) {$k$};
        \node[above] at (1.4,2.1) {$k+q_2$};
        
        \node[below] at (-2.1,-0.1) {$k-r_1$};
        \node[below] at (0,-0.1) {$k$};
        \node[below] at (2.1,-0.1) {$k+r_2$};
        
        \node[left] at (-3.6,1.5) {$(x,u)$};
        \node[left] at (-3.9,0) {$y$};
        
        \draw[-stealth,thick] (-3.3,1.3) -- (-3.3,0.2);
        \node[left] at (-3.3,0.75) {$\varphi$};
        
        \draw[-stealth,thick] (3.3,0.2) -- (3.3,1.3);
        \node[right] at (3.3,0.75) {$F$};
        
        \end{tikzpicture}
        \caption{One-to-one correspondence of the trajectories}
        \label{fig:One_to_one_correspondence}
    \end{figure}

    In the following, we recall a slightly modified yet equivalent formulation as stated in \cite{SchrotshamerKolarSchoeberl:2025}, along with the two special cases of forward- and backward-flatness, which naturally arise when the flat output and the flat parameterization are restricted to forward- or backward-shifts, respectively.
    \vspace{0.05cm}

	\begin{defn}
		\label{def:flatness}The system \eqref{eq:sysEq} is said to be flat around an equilibrium $(x_{0},u_{0})$, if the $n+m$ coordinate functions $x$ and $u$ can be expressed locally by an $m$-tuple of functions
		\begin{equation}
			y^{j}=\varphi^{j}(\zeta_{[-Q_1,-1]},x,u,u_{[1,Q_2]})\,,\quad\,j=1,\ldots,m\label{eq:flat_output}
		\end{equation}
		called the flat output, and its backward- and forward-shifts up to some finite order, i.e.
    	\begin{equation}\label{eq:flat_param}
    		\begin{aligned}
    			x&=F_{x}(y_{[-R_1,R_2-1]})\,,\\
    			\:u&=F_{u}(y_{[-R_1,R_2]})\,.
    		\end{aligned}
    	\end{equation}
        \vspace{-0.2cm}
	\end{defn}
    For a given flat output \eqref{eq:flat_output}, the parameterization \eqref{eq:flat_param} of the state- and input variables $x$ and $u$ is unique. 
    The special case of forward-flatness follows directly from Definition \ref{def:flatness} when no backward-shifts occur in the flat output \eqref{eq:flat_output} and $R_1=(0,\ldots,0)$ applies for the parameterization \eqref{eq:flat_param}.
    \vspace{0.05cm}
	\begin{defn}
		\label{def:forward-flatness}The system \eqref{eq:sysEq} is said
		to be forward-flat around an equilibrium $(x_{0},u_{0})$, if the
		$n+m$ coordinate functions $x$ and $u$ can be expressed locally
		by an $m$-tuple of functions
		\begin{equation}
			y^{j}=\varphi^{j}(x,u,u_{[1,Q_2]})\,,\quad j=1,\ldots,m\label{eq:forward-flat_output}
		\end{equation}
		called the forward-flat output, and its forward-shifts up to some finite order, i.e.
    	\begin{equation*}
    		\begin{aligned}
    			x&=F_{x}(y_{[0,R_2-1]})\,,\\
    			\:u&=F_{u}(y_{[0,R_2]})\,.
    		\end{aligned}
    	\end{equation*}
        \vspace{-0.2cm}
	\end{defn}

    Conversely, if no forward-shifts appear in the flat output \eqref{eq:flat_output} and $R_2=(0,\ldots,0)$ applies for the parameterization \eqref{eq:flat_param}, then Definition \ref{def:flatness} leads to the subclass of backward-flat systems.
    \vspace{0.05cm}
	\begin{defn}
		\label{def:backward-flatness}The system \eqref{eq:sysEq} is said
		to be backward-flat around an equilibrium $(x_{0},u_{0})$, if the
		$n+m$ coordinate functions $x$ and $u$ can be expressed locally
		by an $m$-tuple of functions
		\begin{equation}
			y^{j}=\varphi^{j}(\zeta_{[-Q_1,-1]},x,u)\,,\quad j=1,\ldots,m\label{eq:backward-flat_output}
		\end{equation}
		called the backward-flat output, and its backward-shifts up to some finite order, i.e.
        \begin{equation*}
    		\begin{aligned}
    			x&=F_{x}(y_{[-R_1,-1]})\,,\\
    			\:u&=F_{u}(y_{[-R_1,0]})\,.
    		\end{aligned}
	    \end{equation*}
        \vspace{-0.2cm}
	\end{defn}

    \subsection{Geometric Test for Forward-Flatness}
    Since our algorithmic approach for checking the backward-flatness of discrete-time systems builds on the geometric test for forward-flatness of \cite{KolarDiwoldSchoberl:2023}, we first recall some basic differential geometric concepts that are employed therein. For a general introduction to differential geometry, we refer, e.g., to \cite{Boothby:1986} or \cite{Lee:2012}.
    
    Let $\mathcal{M}$ be an $n$-dimensional manifold with local coordinates $x^{1},\ldots,x^{n}$. A vector field $v$ on $\mathcal{M}$ has the form $v=v^{i}(x)\partial_{x^{i}}$, where $\partial_{x^{i}}$, $i=1,\ldots,n$ denotes the basis vector fields, and $v^{i}(x)$, $i=1,\ldots,n$ are smooth functions. A $d$-dimensional distribution on $\mathcal{M}$ is denoted by $D=\mathrm{span}\{v_{1},\ldots,v_{d}\}$, where $v_{1},\ldots,v_{d}$ are linearly independent vector fields.\\
    Geometrically, the map $f$ of \eqref{eq:sysEq} can be interpreted as a map from a manifold $\mathcal{X}\times\mathcal{U}$ with coordinates $(x,u)$ to a manifold $\mathcal{X}^{+}$ with coordinates $x^{+}$, i.e. 
    \begin{equation}
    	f:\mathcal{X}\times\mathcal{U}\rightarrow\mathcal{X}^{+}\,.\label{eq:map_f}
    \end{equation}
    Additionally, the map $\pi:\mathcal{X}\times\mathcal{U}\rightarrow\mathcal{X}^{+}$ is defined by
    \begin{equation*}
    	x^{i,+}=x^{i},\hphantom{aa}i=1,\dots,n\,.\label{eq:piEq}
    \end{equation*}
    A vector field $v$ is called ''projectable'' w.r.t. the map \eqref{eq:map_f}, if a pointwise application of the tangent map $f_*:\mathcal{T}(\mathcal{X}\times\mathcal{U})\rightarrow\mathcal{T}(\mathcal{X}^+)$ induces a well-defined vector field $w=f_*(v)$ on $\mathcal{X}^+$, with $f_*(v)$ being called the pushforward of the vector field $v$. Moreover, we call a distribution $D$ on $\mathcal{X}\times\mathcal{U}$ projectable if it admits a basis consisting of projectable vector fields only. \\
    With these geometric preliminaries we can give a modified but equivalent formulation of Algorithm 1 in \cite{KolarDiwoldSchoberl:2023}.
    
    \begin{alg}\label{alg:definition_sequence}Start
    	with the distribution $E_{0}=\mathrm{span}\{\partial_{u^1},$ $\ldots,\partial_{u^m}\}$ on $\mathcal{X}\times\mathcal{U}$
    	and repeat the following steps for $k\geq1$:\textbf{\vspace{-0.2cm}
    	}
    	\begin{enumerate}
    		\item[\emph{1.}] Determine the largest projectable subdistribution $D_{k-1}$ that
    		meets $D_{k-1}\subseteq E_{k-1}$.
    		\item[\emph{2.}] Compute the pushforward $\Delta_{k}=f_{*}(D_{k-1})$.
    		\item[\emph{3.}] Proceed with $E_{k}=\pi_{*}^{-1}(\Delta_{k})$.\textbf{\vspace{-0.1cm}
    			\vspace{-0.1cm}
    		}
    	\end{enumerate}
    	Stop if $\dim(E_{\bar{k}})=\dim(E_{\bar{k}-1})$ for some step $\bar{k}$.
    \end{alg}


    The sequence of distributions defined by Algorithm \ref{alg:definition_sequence} allows to check the forward-flatness of a system \eqref{eq:sysEq} in a straightforward way as shown in \cite{KolarDiwoldSchoberl:2023}.
    
    \begin{thm}
    	\label{thm:condition}The system \eqref{eq:sysEq} is forward-flat if and only if $\dim(E_{\bar{k}-1})=n+m$.
    \end{thm}

    Since the sequence of distributions generated by Algorithm \ref{alg:definition_sequence} generalizes the sequence of distributions used in the static feedback linearization test in \cite{Grizzle:1986}, the latter is included as a special case.

    \begin{rem}
        It should be noted that in \cite{KolarSchrotshamerSchoberl:2025} a dual geometric test for forward-flatness based on a unique sequence of integrable codistributions has been introduced.
        Consequently, in the remainder of this paper, the forward-flatness of discrete-time systems could equivalently be tested using the codistribution-based approach proposed in \cite{KolarSchrotshamerSchoberl:2025}.
    \end{rem}

\section{Associated Flat Systems} \label{sec:three}
    Given a discrete-time system of the form \eqref{eq:sysEq_ext}, in this section, we first derive an associated system whose trajectories are in a one-to-one correspondence with those of the original system \eqref{eq:sysEq_ext}. We then show that the associated system is flat if and only if the original system is flat and discuss this relation for the special cases of forward- and backward-flatness.

    \begin{thm} \label{thm:time_mirrored_system}
        The trajectories of the discrete-time system
        \begin{equation} \label{eq:flat_sys_mirrored}
            \begin{aligned}
                z^+&=\psi_x(z,v)\\
                \eta&=\psi_u(z,v)\
            \end{aligned}
        \end{equation}
    with the functions $\psi_x$ and $\psi_u$ from \eqref{eq:sysEq_ext_inv} are related to the trajectories of the discrete-time system \eqref{eq:sysEq_ext} for any fixed time step $k$ via the one-to-one correspondence
        \begin{equation}\label{eq:correspondence}
            \begin{aligned}
                x(k+l)&=z(k-l+1)\\
                u(k+l)&=\eta(k-l)\\
                \zeta(k+l)&=v(k-l)
            \end{aligned}
        \end{equation}
        with $l\in\mathbb{Z}$.
    \end{thm}
        
    \begin{pf}
        By substituting a trajectory of the system \eqref{eq:sysEq_ext} into the inverse map \eqref{eq:sysEq_ext_inv}, we obtain for a fixed time step $k$ the relations
        \begin{equation*}
            \begin{aligned}
                x(k+l) &= \psi_x(x(k+(l+1)),\zeta(k+l)) \\
                u(k+l) &= \psi_u(x(k+(l+1)),\zeta(k+l))
            \end{aligned}
        \end{equation*}
        with $l\in\mathbb{Z}$.
        Applying the one-to-one correspondence \eqref{eq:correspondence} yields
        \begin{equation} \label{eq:sysEq_assoz_seq}
            \begin{aligned}
                z(k-l+1) &= \psi_x(z(k-(l+1)+1),v(k-l)) \\
                \eta(k-l) &= \psi_u(z(k-(l+1)+1),v(k-l)) \,,
            \end{aligned}
        \end{equation}
        which shows that the system equations of the associated system \eqref{eq:flat_sys_mirrored} are satisfied, thereby completing the proof.
    \end{pf}
    \begin{rem} \label{rem:assoc_of_assoc}
        Due to the fact that the trajectories of a system \eqref{eq:sysEq_ext} and its associated system \eqref{eq:flat_sys_mirrored} are related by the one-to-one correspondence \eqref{eq:correspondence}, it follows directly that the associated system of \eqref{eq:flat_sys_mirrored} coincides with the original system \eqref{eq:sysEq_ext}.
    \end{rem}

    It should be noted that Theorem \ref{thm:time_mirrored_system} applies to all discrete-time nonlinear systems of the form \eqref{eq:sysEq_ext} regardless of their flatness.
    The following theorem addresses the relation between the flatness properties of the systems \eqref{eq:sysEq_ext} and \eqref{eq:flat_sys_mirrored}.
    
    \begin{thm} \label{thm:mirrored}
        The associated system \eqref{eq:flat_sys_mirrored} of Theorem \ref{thm:time_mirrored_system} is flat according to Definition \ref{def:flatness} if and only if the original system \eqref{eq:sysEq_ext} is also flat.
        Given a flat output \eqref{eq:flat_output} of the system \eqref{eq:sysEq_ext}, the associated system \eqref{eq:flat_sys_mirrored} possesses the flat output
        \begin{equation} \label{eq:flat_output_mirrored}
        \begin{aligned}
            \hat{y}&=\varphi(v_{[Q_1,1]},\psi_x(z,v),\psi_u(z,v),\eta_{[-1,-Q_2]}) \\
            &=\hat \varphi(\eta_{[-Q_2,-1]},z,v,v_{[1,Q_1]}) \,,
        \end{aligned}
        \end{equation}
        whose trajectories are related to those of \eqref{eq:flat_output} for any fixed time step $k$ via the one-to-one correspondence\footnote{Equation \eqref{eq:y_mirror} does of course not impose any constraints on the choice of the trajectory of $\hat y$.}
        \begin{equation} \label{eq:y_mirror}
            \hat y(k+l) = y(k-l) \,.
        \end{equation}
        The parameterization of the system variables of \eqref{eq:flat_sys_mirrored} by the flat output \eqref{eq:flat_output_mirrored} follows as
        \begin{equation} \label{eq:flat_param_mirrored}
            \begin{aligned}
                z&=F_x(\hat{y}_{[R_1-1,-R_2]})=F_z(\hat{y}_{[-R_2,R_1-1]}) \\
                v&=g(F_x(\hat{y}_{[R_1,-R_2+1]}),F_u(\hat{y}_{[R_1,-R_2]}))=F_v(\hat{y}_{[-R_2,R_1]}) \,,
            \end{aligned}
        \end{equation}
        with the maps $F_x$ and $F_u$ from \eqref{eq:flat_param}.
    \end{thm}

    \begin{pf} 
        \\ \eqref{eq:sysEq_ext} is flat $\Rightarrow$ \eqref{eq:flat_sys_mirrored} is flat:
        If the system \eqref{eq:sysEq_ext} is flat in the sense of Definition \ref{def:flatness} with the maps \eqref{eq:flat_output} and \eqref{eq:flat_param}, the one-to-one correspondence between the trajectories of the system variables $(x,u,\zeta)$ and the flat output $y$ at a fixed time step $k$ is given by
        \begin{equation}\label{eq:proof_output_orig}
        \scalemath{0.95}{
        \begin{aligned}
            y(k+l)=&\varphi(\zeta(k+(l-Q_1)),\ldots,\zeta(k+(l-1)),x(k+l),\\
            &u(k+l),u(k+(l+1)),\ldots,u(k+(l+Q_2)))
        \end{aligned}}
        \end{equation}
        and 
        \begin{equation}\label{eq:proof_param_orig}
        \scalemath{0.95}{
        \begin{aligned}
            x(k+l)&=F_x(y(k+(l-R_1)),\ldots,y(k+(l+R_2-1))) \\
            u(k+l)&=F_u(y(k+(l-R_1)),\ldots,y(k+(l+R_2))) \\
            \zeta(k+l)&=g(F_x(y(k+(l-R_1)),\ldots,y(k+(l+R_2-1))),\\
            &\quad F_u(y(k+(l-R_1)),\ldots,y(k+(l+R_2)))) \,.
        \end{aligned}}
        \end{equation}
        Applying the one-to-one correspondences \eqref{eq:correspondence} and \eqref{eq:y_mirror} to \eqref{eq:proof_output_orig} and \eqref{eq:proof_param_orig} yields
        \begin{equation} \label{eq:hat_y_seq}
        \scalemath{0.95}{
        \begin{aligned}
            \hat y(k-l)=&\varphi(v(k-(l-Q_1)),\ldots,v(k-(l-1)),z(k-l+1),\\
            &\eta(k-l),\eta(k-(l+1)),\ldots,\eta(k-(l+Q_2)))
        \end{aligned}}
        \end{equation}
        and
        \begin{equation} \label{eq:23}
        \scalemath{0.94}{
        \begin{aligned}
            z(k-l+1)&=F_x(\hat y(k-(l-R_1)),\ldots,\hat y(k-(l+R_2-1))) \\
            \eta(k-l)&=F_u(\hat y(k-(l-R_1)),\ldots,\hat y(k-(l+R_2))) \\
            v(k-l)&=g(F_x(\hat y(k-(l-R_1)),\ldots,\hat y(k-(l+R_2-1))),\\
            &\quad F_u(\hat y(k-(l-R_1)),\ldots,\hat y(k-(l+R_2))))\,.
        \end{aligned}}
        \end{equation}
        Finally, substituting \eqref{eq:sysEq_assoz_seq} into \eqref{eq:hat_y_seq} gives
        \begin{equation*}
        \scalemath{0.95}{
        \begin{aligned}
            \hat y(k-l)=&\varphi(v(k-l+Q_1),\ldots,v(k-l+1),\\
            &\psi_x(z(k-l),v(k-l)),\psi_u(z(k-l),v(k-l)),\\
            &\eta(k-l-1),\ldots,\eta(k-l-Q_2)) \,,
        \end{aligned}}
        \end{equation*}
        which coincides with \eqref{eq:flat_output_mirrored}. Moreover, by rewriting the first equation of \eqref{eq:23} as         
        \begin{equation*}
            \scalemath{0.95}{
            z(k-l)=F_x(\hat y(k-l+R_1-1),\ldots,\hat y(k-l-R_2)) \,,}
        \end{equation*}
        we can immediately observe the equivalence of \eqref{eq:23} with \eqref{eq:flat_param_mirrored}.
        Hence, \eqref{eq:flat_output_mirrored} is indeed a flat output of the associated system \eqref{eq:flat_sys_mirrored} with the corresponding parameterization \eqref{eq:flat_param_mirrored}.
        
        \eqref{eq:flat_sys_mirrored} is flat $\Rightarrow$ \eqref{eq:sysEq_ext} is flat: With Remark \ref{rem:assoc_of_assoc}, it is evident that the flatness of \eqref{eq:flat_sys_mirrored} with a flat output \eqref{eq:flat_output_mirrored} also implies the flatness of \eqref{eq:sysEq_ext} with a flat output
        \begin{equation} \label{eq:flat_output_other_dir}
        \begin{aligned}
            y&=\hat \varphi(u_{[Q_2,1]},f(x,u),g(x,u),\zeta_{[-1,-Q_1]})\\
            &=\varphi(\zeta_{[-Q_1,-1]},x,u,u_{[1,Q_2]}) \,.
        \end{aligned}
        \end{equation}
        The relation between the flatness properties of \eqref{eq:sysEq_ext} and \eqref{eq:flat_sys_mirrored} is illustrated in Fig. \ref{fig:corespondence_y_inv}.
        
        \begin{figure}[h!]
            \centering
            \begin{tikzpicture}[scale=1.5]

            \node at (0,0) {$(x,u,\zeta)$};
    
            \node at (0,-1.5) {$(y)$};

            \node at (1.9,0) {$(z,v,\eta)$};
    
            \node at (1.9,-1.5) {$(\hat y)$};
    
            \draw[-stealth,thick] (-0.1,-1.2) -- (-0.1,-0.3);
            \node[left] at (-0.1,-0.75) {\eqref{eq:flat_param}};

            \draw[-stealth,thick] (0.1,-0.3) -- (0.1,-1.2);
            \node[right] at (0.1,-0.75) {\eqref{eq:flat_output}};

            \draw[stealth-stealth,thick] (1.4,0) -- (0.5,0);
            \node[above] at (0.95,0) {\eqref{eq:correspondence}};

            \draw[stealth-stealth,thick] (0.5,-1.5) -- (1.4,-1.5);
            \node[above] at (0.95,-1.5) {\eqref{eq:y_mirror}};

            \draw[-stealth,thick] (1.8,-1.2) -- (1.8,-0.3);
            \node[left] at (1.8,-0.75) {\eqref{eq:flat_param_mirrored}};

            \draw[-stealth,thick] (2,-0.3) -- (2,-1.2);
            \node[right] at (2,-0.75) {\eqref{eq:flat_output_mirrored}};

            \end{tikzpicture}
            \caption{Relation between the flatness properties of the systems \eqref{eq:sysEq_ext} and \eqref{eq:flat_sys_mirrored}}
            \label{fig:corespondence_y_inv}
        \end{figure}
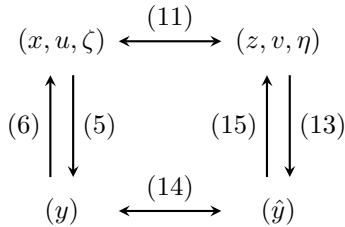
    \end{pf}

    In the following, we will apply the results of Theorem \ref{thm:mirrored} to the subclasses of forward- and backward-flat systems according to Definitions \ref{def:forward-flatness} and \ref{def:backward-flatness}.
    \begin{cor} \label{cor:flatness_of_mirrored_system}
        Consider a system \eqref{eq:sysEq_ext} that is flat according to Definition \ref{def:flatness} with a flat output \eqref{eq:flat_output} and the corresponding parameterizing map \eqref{eq:flat_param}.
        For the special cases of forward- and backward-flatness, the following applies:
        \begin{enumerate}[(i)]
            \item The associated system \eqref{eq:flat_sys_mirrored} is backward-flat with a backward-flat output $\hat{y}=\varphi(\psi_x(z,v),\psi_u(z,v),$ $\eta_{[-1,-Q_2]})=\hat \varphi(\eta_{[-Q_2,-1]},z,v)$ if and only if the original system is forward-flat according to Definition \ref{def:forward-flatness} with a forward-flat output \eqref{eq:forward-flat_output}.
            \item The associated system \eqref{eq:flat_sys_mirrored} is forward-flat with a forward-flat output $\hat{y}=\varphi(v_{[Q_1,1]},\psi_x(z,v),\psi_u(z,v))$ $=\hat \varphi(z,v,v_{[1,Q_1]})$ if and only if the original system is backward-flat according to Definition \ref{def:backward-flatness} with a backward-flat output \eqref{eq:backward-flat_output}.
        \end{enumerate}
    \end{cor}

\section{Testing backward-flatness}  \label{sec:four}

    \afterpage{
    \begin{figure*}[!b]
        \begin{equation*}\label{eq:ex2_Jacobian}
        \scalemath{0.83}{
        \begin{aligned}
            &\partial_{y_{[-R_1,0]}}\left[\begin{array}{c}F_{x}(y_{[-R_1,-1]})\\F_{u}(y_{[-R_1,0]})\\g(F_x(y_{[-R_1,-1]}),F_u(y_{[-R_1,0]}))\end{array}\right]=
            &\left[\begin{array}{ccccccc}
            \usegray\tikzmarkin{r}(0.9,-0.1)(-1.2,0.35) 0 & 0 & 1 & \usegraydark\tikzmarkin{t}(1.35,-0.1)(-0.8,0.35) 0 & 0 & 0 & 0 \\
            \tfrac{-y^2_{[-1]}+y^2_{[-2]}}{\left(y^1_{[-1]}+y^1_{[-3]}\right)^2} & -\tfrac{1}{y^1_{[-1]}+y^1_{[-3]}} & 0 & \tfrac{1}{y^1_{[-1]}+y^1_{[-3]}} & \tfrac{-y^2_{[-1]}+y^2_{[-2]}}{\left(y^1_{[-1]}+y^1_{[-3]}\right)^2} & 0 & 0 \\
            \tfrac{\left(y^2_{[-1]}-y^2_{[-2]}\right) y^1_{[-1]}}{\left(y^1_{[-1]}+y^1_{[-3]}\right)^2} & \tfrac{y^1_{[-1]}}{y^1_{[-1]}+y^1_{[-3]}} & 0 & \tfrac{y^1_{[-3]}}{y^1_{[-1]}+y^1_{[-3]}} & -\tfrac{y^1_{[-3]}\left(y^2_{[-1]}-y^2_{[-2]}\right)}{\left(y^1_{[-1]}+y^1_{[-3]}\right)^2} & 0 & 0 \\
            0 & 0 \tikzmarkend{r} & 0 & 0 & 1 \tikzmarkend{t} & 0 & 0 \\
            0 & 0 & 0 & 0 & 0 & \usegraydark\tikzmarkin{u}(0.1,-0.4)(-0.65,0.35) 0 & 1 \\
            0 & 0 & \tfrac{-y^2_{[0]}+y^2_{[-1]}}{\left(y^1_{[0]}+y^1_{[-2]}\right)^2} & -\tfrac{1}{y^1_{[0]}+y^1_{[-2]}} & 0 & \tfrac{1}{y^1_{[0]}+y^1_{[-2]}} & \tfrac{-y^2_{[0]}+y^2_{[-1]}}{\left(y^1_{[0]}+y^1_{[-2]}\right)^2} \tikzmarkend{u} \\
            \usegray\tikzmarkin{s}(0.1,-0.45)(-1,0.35) 0 & 0 & 1 & 0 & 0 & 0 & 0 \\
            \tfrac{-y^2_{[-1]}+y^2_{[-2]}}{\left(y^1_{[-1]}+y^1_{[-3]}\right)^2} & -\tfrac{1}{y^1_{[-1]}+y^1_{[-3]}} \tikzmarkend{s} & 0 & \tfrac{1}{y^1_{[-1]}+y^1_{[-3]}} & \tfrac{-y^2_{[-1]}+y^2_{[-2]}}{\left(y^1_{[-1]}+y^1_{[-3]}\right)^2} & 0 & 0
            \end{array}\right]
        \end{aligned}
        }
        \end{equation*}
        \caption{Jacobian matrix \eqref{eq:Jacobian_ext} of the backward-flat system \eqref{eq:sysEq_ex1}}
        \label{fig:Jacobian1_ex}
    \end{figure*}
    \begin{figure*}[!b]
        \begin{equation*}\label{eq:ex2_Jacobian_mirrored}
            \scalemath{0.83}{
            \begin{aligned}
                &\partial_{\hat y_{[0,R_1]}}\left[\begin{array}{c}F_{x}(\hat y_{[R_1-1,0]})\\ g(F_x(\hat y_{[R_1,1]}),F_u(\hat y_{[R_1,0]}))\\F_{u}(\hat y_{[R_1,0]})\end{array}\right]=
                &\left[\begin{array}{ccccccc}
                \usegraydark\tikzmarkin{v}(0.65,-0.1)(-1.05,0.35) 0 & 0 & 1 & \usegray\tikzmarkin{x}(0.9,-0.1)(-0.7,0.35) 0 & 0 & 0 & 0 \\
                \tfrac{-\hat y^2_{[0]}+\hat y^2_{[1]}}{\left(\hat y^1_{[0]}+\hat y^1_{[2]}\right)^2} & \tfrac{1}{\hat y^1_{[0]}+\hat y^1_{[2]}} & 0 & -\tfrac{1}{\hat y^1_{[0]}+\hat y^1_{[2]}} & \tfrac{-\hat y^2_{[0]}+\hat y^2_{[1]}}{\left(\hat y^1_{[0]}+\hat y^1_{[2]}\right)^2} & 0 & 0 \\
                -\tfrac{\hat y^1_{[2]}\left(\hat y^2_{[0]}-\hat y^2_{[1]}\right)}{\left(\hat y^1_{[0]}+\hat y^1_{[2]}\right)^2} & \tfrac{\hat y^1_{[2]}}{\hat y^1_{[0]}+\hat y^1_{[2]}} & 0 & \tfrac{\hat y^1_{[0]}}{\hat y^1_{[0]}+\hat y^1_{[2]}} & \tfrac{\left(\hat y^2_{[0]}-\hat y^2_{[1]}\right) \hat y^1_{[0]}}{\left(\hat y^1_{[0]}+\hat y^1_{[2]}\right)^2} & 0 & 0 \\
                1 & 0 \tikzmarkend{v} & 0 & 0 & 0 \tikzmarkend{x} & 0 & 0 \\
                0 & 0 & 0 & 0 & 1 & \usegray\tikzmarkin{y}(0.1,-0.45)(-0.7,0.35) 0 & 0 \\
                0 & 0 & \tfrac{-\hat y^2_{[1]}+\hat y^2_{[2]}}{\left(\hat y^1_{[1]}+\hat y^1_{[3]}\right)^2} & \tfrac{1}{\hat y^1_{[1]}+\hat y^1_{[3]}} & 0 & -\tfrac{1}{\hat y^1_{[1]}+\hat y^1_{[3]}} & \tfrac{-\hat y^2_{[1]}+\hat y^2_{[2]}}{\left(\hat y^1_{[1]}+\hat y^1_{[3]}\right)^2} \tikzmarkend{y} \\
                \usegraydark\tikzmarkin{w}(0.1,-0.45)(-0.8,0.35) 1 & 0 & 0 & 0 & 0 & 0 & 0 \\
                \tfrac{-\hat y^2_{[0]}+\hat y^2_{[1]}}{\left(\hat y^1_{[0]}+\hat y^1_{[2]}\right)^2} & \tfrac{1}{\hat y^1_{[0]}+\hat y^1_{[2]}} \tikzmarkend{w} & 0 & -\tfrac{1}{\hat y^1_{[0]}+\hat y^1_{[2]}} & \tfrac{-\hat y^2_{[0]}+\hat y^2_{[1]}}{\left(\hat y^1_{[0]}+\hat y^1_{[2]}\right)^2} & 0 & 0
            \end{array}\right]
            \end{aligned}
            }
        \end{equation*}
        \caption{Jacobian matrix \eqref{eq:Jacobian_ext} of the forward-flat associated system \eqref{eq:sysEq_ex1_mirrored}}
        \label{fig:Jacobian2_ex}
    \end{figure*}}

    \subsection{An Algorithmic Test for Backward-Flatness}
    Based on item (ii) of Corollary \ref{cor:flatness_of_mirrored_system}, the backward-flatness of a discrete-time system \eqref{eq:sysEq_ext} can be tested indirectly by checking the associated system \eqref{eq:flat_sys_mirrored} for forward-flatness.\\
    The forward-flatness of the associated system can be systematically checked by computing the sequence of distributions generated by Algorithm \ref{alg:definition_sequence} and applying Theorem \ref{thm:condition}. In case of a positive result, a forward-flat output can be derived based on these distributions, as demonstrated in \cite{KolarDiwoldSchoberl:2023}. The corresponding backward-flat output of the original system follows from \eqref{eq:flat_output_other_dir}.
    This systematic approach for testing backward-flatness and deriving a backward-flat output is summarized in the following algorithm.
    
    \begin{alg}\label{alg:testing_backward_flatness} Given a discrete-time system \eqref{eq:sysEq_ext}:\textbf{\vspace{-0.2cm}
    	}
    	\begin{enumerate}
    		\item Calculate the associated system \eqref{eq:flat_sys_mirrored}.
            \item Check \eqref{eq:flat_sys_mirrored} for forward-flatness by applying Algorithm \ref{alg:definition_sequence} and Theorem \ref{thm:condition}.
            \begin{enumerate}
                \item[(a)] Stop if \eqref{eq:flat_sys_mirrored} is not forward-flat. With item (ii) of Corollary \ref{cor:flatness_of_mirrored_system} it follows that the original system is not backward-flat.
    		      \item[(b)] If \eqref{eq:flat_sys_mirrored} is forward-flat, a forward-flat output $\hat y=\hat \varphi(z)$ can be derived as in \cite{KolarDiwoldSchoberl:2023}.\footnote{According to Theorem 12 in \cite{KolarDiwoldSchoberl:2021}, every forward-flat system admits a flat output which only depends on the state variables.} According to item (ii) of Corollary \ref{cor:flatness_of_mirrored_system}, the original system is backward-flat.
            \end{enumerate}
            \item[(3)] With $\hat y=\hat \varphi(z)$ and \eqref{eq:flat_output_other_dir}, we obtain $y=\hat\varphi(f(x,u))=\varphi(x,u)$.
    	\end{enumerate}
    \end{alg}
    \vspace{1ex}
    \begin{rem}
        It should be noted that in contrast to the forward-flat output $\hat y=\hat \varphi(z)$, which depends only on the state variables, the corresponding backward-flat output $y=\varphi(x,u)$ depends both on state and input variables. In other words, every backward-flat system admits an $(x,u)$-flat output.
    \end{rem}
    
    With the following academic example, we illustrate our results by applying Algorithm \ref{alg:testing_backward_flatness} to test a discrete-time system for backward-flatness and to derive a corresponding backward-flat output.

    \begin{exmp} \label{ex:academic}
    Consider the system
    \begin{equation} \label{eq:sysEq_ex1}
		\begin{aligned}
			x^{1,+} &= x^4 \\
			x^{2,+} &= u^2 \\
			x^{3,+} &= x^3+x^2 x^4+x^1 u^2 \\
			x^{4,+} &= u^1
		\end{aligned}
	\end{equation}
    with $n=4$, $m=2$, and $g^1(x,u)=x^1,\,g^2(x,u)=x^2$.
    First, one would check whether the system is forward-flat by using Algorithm \ref{alg:definition_sequence}, which yields the sequences
    \begin{align}
	E_{0} & =\mathrm{span}\left\{ \partial_{u^{1}},\partial_{u^{2}}\right\} \nonumber \\
	E_{1} & =\mathrm{span}\left\{ \partial_{x^{4}},\partial_{u^{1}},\partial_{u^{2}}\right\} \nonumber \\
    E_{2} & =\mathrm{span}\left\{ \partial_{x^{4}},\partial_{u^{1}},\partial_{u^{2}}\right\} \nonumber 
    \end{align}
    and
    \begin{equation*}
    	D_{0}=\mathrm{span}\left\{ \partial_{u^{1}}\right\} \,,\:D_{1}=\mathrm{span}\left\{ \partial_{u^{1}}\right\}
    \end{equation*}
    and therefore stops at $k=2$. Since $\operatorname{dim}(E_2)=\operatorname{dim}(E_1)$ and $\operatorname{dim}(E_2)\neq n+m$, the system \eqref{eq:sysEq_ex1} is not forward-flat.\\
    As the test for forward-flatness did not yield a positive result, next we can check whether the system is backward-flat by using Algorithm \ref{alg:testing_backward_flatness}.
    
    Step 1) The associated system is obtained as described in Section \ref{sec:three} by first computing the inverse of \eqref{eq:sysEq_ex1}, i.e.
    \begin{equation*}
        \begin{aligned}
			x^{1} & = \zeta^1 \\
			x^{2} & = \zeta^2 \\
			x^{3} & = x^{3,+}-x^{1,+} \zeta^2-x^{2,+} \zeta^1 \\
			x^{4} & = x^{1,+}
		\end{aligned}
        \hspace{5ex}
        \begin{aligned}
			u^1 & = x^{4,+} \\
            u^2 & = x^{2,+}
		\end{aligned}
	\end{equation*}
    with $x=\psi_x(x^+,\zeta)$, $u=\psi_u(x^+,\zeta)$. Using \eqref{eq:flat_sys_mirrored}, we obtain the associated system
    \begin{equation}\label{eq:sysEq_ex1_mirrored}
        \begin{aligned}
			z^{1,+} & = v^1 \\
			z^{2,+} & = v^2 \\
			z^{3,+} & = z^{3}-z^{1} v^2-z^{2} v^1 \\
			z^{4,+} & = z^{1}\,.
		\end{aligned}
        \hspace{5ex}
        \begin{aligned}
			\eta^1 & = z^{4} \\
            \eta^2 & = z^{2}
		\end{aligned}
	\end{equation}
   
    Step 2) Applying Algorithm~\ref{alg:definition_sequence} yields the sequence of involutive distributions
    \begin{align}
	E_{0} & =\mathrm{span}\left\{ \partial_{v^{1}},\partial_{v^{2}}\right\} \nonumber \\
	E_{1} & =\mathrm{span}\left\{ \partial_{z^{2}}-z^4\partial_{z^{3}},\partial_{v^{1}},\partial_{v^{2}}\right\} \label{eq:E_sequ_ex}\\
	E_{2} & =\mathrm{span}\left\{ \partial_{z^{3}},\partial_{z^{1}},\partial_{z^{2}},\partial_{v^{1}},\partial_{v^{2}}\right\} \nonumber \\
	E_{3} & =\mathrm{span}\left\{ \partial_{z^{3}},\partial_{z^{4}},\partial_{z^{1}},\partial_{z^{2}},\partial_{v^{1}},\partial_{v^{2}}\right\} \nonumber 
    \end{align}
    and the sequence of the largest projectable subdistributions 
    \begin{equation*}
    	D_{0}=\mathrm{span}\left\{ \partial_{v^{2}}\right\} \,,\:D_{1}=E_{1}\,,\:D_{2}=E_{2}\,.\label{eq:D_sequ_ex}
    \end{equation*}
    According to Theorem \ref{thm:condition}, because of $\mathrm{dim}(E_{3})=n+m$ the system \eqref{eq:sysEq_ex1_mirrored} is forward-flat. Consequently, by item (ii) of Corollary \ref{cor:flatness_of_mirrored_system}, the original system is backward-flat.\\
    A possible forward-flat output of \eqref{eq:sysEq_ex1_mirrored} can be determined as demonstrated in \cite{KolarDiwoldSchoberl:2023} by straightening out the sequence of distributions \eqref{eq:E_sequ_ex} with the state transformation $\bar z^1_3=z^4,\, \bar z^1_2=z^3+z^2 z^4,\, \bar z^2_2=z^1,\, \bar z^1_1=z^2$. In these coordinates, the transformed system reads as
    \begin{equation}\label{eq:sysEq_ex1_triangular}
        \begin{aligned}
        \bar z_{3}^{1,+} & = \bar z^2_2\\
        \bar z_{2}^{1,+} & = -\bar z^1_1(v^1+\bar z^1_3)-\bar z^1_2\\
        \bar z_{2}^{2,+} & = v^1\\
        \bar z_{1}^{1,+} & = v^2\,.
        \end{aligned}
    \end{equation}
    Because of the triangular structure of \eqref{eq:sysEq_ex1_triangular}, a possible flat output is given by $\hat{y}=(\bar z^1_3,\bar z^1_2)=(z^4,z^3+z^2 z^4)$.
    
    Step 3) With $\hat{y}=\hat \varphi(z)=(z^4,z^3+z^2 z^4)$ and \eqref{eq:flat_output_other_dir}, the backward-flat output $y$ of the original system \eqref{eq:sysEq_ex1} is obtained as the composition $y=\hat \varphi(f(x,u))$, that is,
    \begin{equation}\label{eq:ex1_y}
        y=\varphi(x,u)=(u^1,x^3+x^2 x^4+u^2(x^1+u^1))\,.
    \end{equation}
    
    We can verify that $y$ is indeed a backward-flat output of the system \eqref{eq:sysEq_ex1} by computing the backward-shifts of the components of the flat output
    \begin{equation*}
        \scalemath{1}{
			\begin{array}{rlrl}
				y_{[-3]}^1 &=\zeta^1_{[-1]} & \hspace{4ex} y_{[-2]}^2 &=x^3-x^2\zeta^1_{[-1]} \\                
				y_{[-2]}^1 &=x^1 & \hspace{4ex} y_{[-1]}^2 &=x^3+x^2 x^4 \\
				y_{[-1]}^1 &=x^4 & \hspace{4ex} y^2 &=x^3+x^2 x^4+u^2(x^1+u^1)\\
				y^1 &= u^1\,.
			\end{array}
            }
		\end{equation*}
    This set of equations can be solved for $x^1,\ldots, x^4, u^1, u^2$ and $\zeta^1_{[-1]}$. Hence, the system \eqref{eq:sysEq_ex1} is indeed backward-flat with the backward-flat output \eqref{eq:ex1_y} and the backward-flat parameterization being of the form $x=F_{x}(y_{[-R_1,-1]})$, $u=F_{u}(y_{[-R_1,0]})$ with $R_1=(3,2)$.
    \end{exmp}

    \subsection{The Jacobian of the Parameterizing Map}
    In \cite{SchrotshamerKolarSchoeberl:2025}, it has been shown that the properties of forward-flatness, backward-flatness and flatness according to Definition \ref{def:flatness} imply rank conditions regarding certain submatrices of the Jacobian of the flat parameterization. These rank conditions are essential for deriving a simple type of linearizing dynamic extension for flat discrete-time systems with two inputs.
    In the following, we show that the main findings of the current contribution --- most notably that the parameterizing maps of a flat system \eqref{eq:sysEq_ext} and its associated system \eqref{eq:flat_sys_mirrored} are related via the one-to-one correspondences \eqref{eq:correspondence} and \eqref{eq:y_mirror} established in Section \ref{sec:three} --- align nicely with the rank conditions presented in \cite{SchrotshamerKolarSchoeberl:2025}.\\
    In Proposition \ref{prop:Jacobi_R2}, we first introduce an additional rank condition, before consolidating these findings in the subsequent corollary.
    
    \begin{prop} \label{prop:Jacobi_R2}
	   For the Jacobian matrix of the parameterizing map \eqref{eq:flat_param} of a system \eqref{eq:sysEq} that is flat according to Definition \ref{def:flatness}, the submatrices $\partial_{y_{[R_2-1]}} F_x$ and $\partial_{y_{[R_2]}} F_u$ have the same rank.
    \end{prop} 
    \begin{pf}
        We start with the identity $\delta\left(F_{x}\right) =f\left(F_{x}, F_{u}\right)$\footnote{The forward- and backward-shift operators, denoted by $\delta$ and $\delta^{-1}$, are defined on spaces with the coordinates $(\ldots,\zeta_{[-1]},x,u,u_{[1]},\ldots)$ and $(\ldots,y_{[-1]},y,y_{[1]},\ldots)$ as in \cite{DiwoldKolarSchoberl:2020}.}, which holds for all flat systems.
        Differentiating both sides w.r.t. $y_{[R_2]}$ and rewriting the resulting equation yields
        \begin{equation*}
        \begin{aligned}
            \partial_{y_{[R_2]}} \delta\left(F_{x}\right) & =\partial_{y_{[R_2]}} f \circ F\\
            \delta(\delta^{-1}(\partial_{y_{[R_2]}} \delta(F_{x})))&=\left(\partial_{u} f \circ F\right)\partial_{y_{[R_2]}} F_{u}\\
        	\delta(\partial_{y_{[R_2-1]}} F_{x}) & =\left(\partial_{u} f \circ F\right) \partial_{y_{[R_2]}} F_{u}
        \end{aligned}
        \end{equation*}
        because the flat parameterization \eqref{eq:flat_param} of the state variables $x$ is independent of the highest forward-shifts $y_{[R_2]}$.
        Since $\operatorname{rank}\left(\partial_{u} f \circ F\right)=\operatorname{rank}(\partial_{u} f)=m$, and the forward-shift does not change the rank, it follows that $\operatorname{rank}(\partial_{y_{[R_2-1]}} F_{x})=\operatorname{rank}(\partial_{y_{[R_2]}} F_{u})$.
    \end{pf}

    The following corollary is an immediate consequence of Proposition \ref{prop:Jacobi_R2} and Propositions 7 and 16 from \cite{SchrotshamerKolarSchoeberl:2025}.
    
    \begin{cor}\label{cor:flatness_rank}
        Consider a flat system \eqref{eq:sysEq_ext} according to Definition \ref{def:flatness} with a flat output \eqref{eq:flat_output}. For the special cases of forward- and backward-flatness, the Jacobian matrix of the flat parameterization \eqref{eq:flat_param} (including $\zeta=g(F_x,F_u)$), i.e.
        \begin{equation}\label{eq:Jacobian_ext}
        \scalemath{0.89}{
        \begin{aligned}
            &\partial_{y_{[-R_1,R_2]}}\left[\begin{array}{c}F_{x}\\[0.5mm] F_{u}\\[0.5mm] g(F_x,F_u)\end{array}\right]=\\
            &\left[\begin{array}{cccc}
    		\partial_{y_{[-R_1]}} F_x & \cdots & \partial_{y_{[R_2-1]}} F_x & \underline{0} \\[1.4mm]
    		\partial_{y_{[-R_1]}} F_u & \cdots & \partial_{y_{[R_2-1]}} F_u & \partial_{y_{[R_2]}} F_u \\[1.4mm]
            \partial_{y_{[-R_1]}} g(F_x,F_u) & \cdots & \partial_{y_{[R_2-1]}} g(F_x,F_u) & \partial_{y_{[R_2]}} g(F_x,F_u)
        	\end{array}\right] \,,
            \end{aligned}
        }
        \end{equation}
        has the following properties:
        \begin{enumerate}[(i)]
            \item If the system is forward-flat according to Definition \ref{def:forward-flatness} with a forward-flat output \eqref{eq:forward-flat_output}, then $\Rank (\partial_{y_{[-R_1]}}F_x)=\Rank (\partial_{y_{[-R_1]}} g(F_x,F_u))=m$ with $R_1=(0,\ldots,0)$ and $\Rank (\partial_{y_{[R_2-1]}}F_x)=\Rank (\partial_{y_{[R_2]}} F_u)<m$.
            \item If the system is backward-flat according to Definition \ref{def:backward-flatness} with a backward-flat output \eqref{eq:backward-flat_output}, then $\Rank (\partial_{y_{[-R_1]}}F_x)=\Rank (\partial_{y_{[-R_1]}} g(F_x,F_u))<m$ and $\Rank (\partial_{y_{[R_2-1]}}F_x)=\Rank (\partial_{y_{[R_2]}} F_u)=m$ with $R_2=(0,\ldots,0)$.
        \end{enumerate}
    \end{cor}

    Below, we illustrate the rank conditions of Corollary \ref{cor:flatness_rank} using the system \eqref{eq:sysEq_ex1} of Example \ref{ex:academic}.
    
    \begin{exmp}
    For the backward-flat system \eqref{eq:sysEq_ex1} with the backward-flat output \eqref{eq:ex1_y}, the extended Jacobian matrix \eqref{eq:Jacobian_ext} of the parameterizing map is illustrated in Fig. \ref{fig:Jacobian1_ex}    
    with $\Rank (\partial_{y_{[-R_1]}}F_x)=\Rank (\partial_{y_{[-R_1]}} g(F_x,F_u))=1<m$ and $\Rank (\partial_{y_{[-1]}}F_x)=\Rank (\partial_{y_{[0]}} F_u)=2=m$. Thus, the rank conditions in item (ii) of Corollary \ref{cor:flatness_rank} are satisfied.\\
    As shown in the proof of Theorem \ref{thm:mirrored}, to obtain the parameterizing map \eqref{eq:flat_param_mirrored} of the associated system, we first substitute \eqref{eq:correspondence} and \eqref{eq:y_mirror} into the flat parameterization \eqref{eq:proof_param_orig}. This substitution effectively mirrors the Jacobian matrix in Fig. \ref{fig:Jacobian1_ex} about the index $k$. Moreover, since the variables $\zeta$ and $u$ are renamed as the new input $v$ and the extension $\eta$, respectively, the $m=2$ bottom-most rows are swapped with the two above them. Finally, by shifting back the expression $z(k-l+1)=F_x(\hat y(k-(l-R_1)),\ldots,\hat y(k-(l+R_2-1)))$ by one yields the flat parameterization \eqref{eq:flat_param_mirrored}. This backward shift is reflected in the Jacobian as well, where the $n=4$ upper-most rows are shifted back by one, resulting in the Jacobian matrix in Fig. \ref{fig:Jacobian2_ex}.
    Due to the fact that the flat parameterization \eqref{eq:flat_param_mirrored} of the associated system is comprised of the functions $F_x$, $F_u$ and $g$ of the parameterizing map \eqref{eq:proof_param_orig}, the rank conditions in item (i) of Corollary \ref{cor:flatness_rank} are satisfied with $\Rank (\partial_{\hat y_{[0]}}F_x)=\Rank (\partial_{\hat y_{[0]}} F_u)=2=m$ and $\Rank (\partial_{\hat y_{[R_1-1]}}F_x)=\Rank (\partial_{\hat y_{[R_1]}} g(F_x,F_u))=1<m$.
    \end{exmp}

\section{Conclusion}
    In this contribution, we have presented a systematic approach for checking discrete-time systems for backward-flatness, which is a special case of difference flatness. Contrary to the special case of forward-flatness, for which necessary and sufficient conditions have been established in \cite{KolarDiwoldSchoberl:2023}, no computationally efficient test for backward-flatness has been available until now.\\
    To address this gap, we first determined an associated system for a given discrete-time system of the form \eqref{eq:sysEq_ext}, whose trajectories are in one-to-one correspondence with those of the original system. Subsequently, we have shown that a discrete-time system is backward-flat if and only if its associated system is forward-flat. Consequently, backward-flatness can be tested indirectly by applying the existing geometric test for forward-flatness to the associated system, and the corresponding backward-flat output of the original system can also be derived systematically. Leveraging the result that every forward-flat system admits an $x$-flat output (Theorem 12 in \cite{KolarDiwoldSchoberl:2021}), we have shown that every backward-flat system admits an $(x,u)$-flat output. Our findings were further substantiated by rank conditions for specific submatrices of the Jacobian matrices of the parameterizations of associated systems.\\
    As the results in the present contribution represent a further step towards a computationally feasible test for difference flatness, future research will focus on establishing necessary and sufficient conditions for the general case according to Definition \ref{def:flatness}.


\bibliography{bibliography}

                                                                       
\end{document}